\documentclass{amsart}
\usepackage{amsmath}
\usepackage{amsfonts}
\usepackage{amssymb}
\usepackage{graphicx}
\usepackage{float, verbatim}
\usepackage{enumerate}
\usepackage{color}
\usepackage{hyperref}

\newcommand{\tr}{\mathbb{T}\mathrm{r}}
\newcommand{\adj}{\mathbb{A}\mathrm{dj}}
\newtheorem{thm}{Theorem}[section]

\newtheorem{cor}[thm]{Corollary}
\newtheorem{defn}[thm]{Definition}

\newtheorem{rem}[thm]{Remark}

\newtheorem{example}[thm]{Example}

\begin{document}
	
	\title{On generalized trigonometric functions and series of rational functions}
	
	\author{Han Yu}
	\address{Han Yu\\
		School of Mathematics \& Statistics\\University of St Andrews\\ St Andrews\\ KY16 9SS\\ UK \\ }
	\curraddr{}
	\email{hy25@st-andrews.ac.uk}
	\thanks{}
	
	\subjclass[2010]{26C15;	33B10}
	
	\keywords{rational functions, generalized trigonometric functions}
	
	\date{}
	
	\dedicatory{hy25@st-andrews.ac.uk}

\begin{abstract}
Here we introduce a way to construct generalized trigonometric functions associated with any complex polynomials, and the well known trigonometric functions can be seen to associate with polynomial $x^2-1$. We will show that those generalized trigonometric functions have algebraic identities which generalizes the well known $\sin^2(x)+\cos^2(x)=1$. One application of the generalized trigonometric functions is evaluating infinite series of rational functions.
\end{abstract}

\maketitle
\section{Introduction}

Trigonometric functions are very commonly used in mathematics. The Euler's identity tells that for all $x\in\mathbb{C}$:
\[
\sin(x)=\frac{e^{i x}-e^{-i x}}{2i} \text{ and } \cos(x)=\frac{e^{i x}+e^{-i x}}{2}.
\]
We see that the trigonometric functions $\sin(x)$ and $\cos(x)$ are just certain linear combinations of exponential functions. From here it is a natural question to ask whether we can find other linear combinations of exponential functions to obtain some functions that are in some sense similar to trigonometric functions. We observe that in the above formulae those exponential functions have exponents $ix, -ix$. The factors $\pm i$ are roots of polynomial $x^2+1=0$ in $\mathbb{C}$. This motivates the following construction:

Let $P\in\mathbb{C}[x]$ be a polynomial of degree $m\geq 1$. Then let $r_1,r_2,\dots,r_m$ be the roots of $P$. We shall construct $m$ functions:
\[
S^{P}_l:\mathbb{C}\to\mathbb{C}, l\in\{0,1,2,\dots,m-1\}.
\] 
The functions have form $S^{P}_l(x)=\sum_{j=1}^{m}T^{P}_{l,j}e^{-i r_j x}$ with $T^{P}_{l,j}\in\mathbb{C}$. 

Now we are going to describe the coefficients $T^{P}_{l,j}$. For any choice of $l$ different numbers say $s_1,...,s_l$ from $1,2...m-1$ we can assign the product
\[
r_{s_1}r_{s_2}...r_{s_l},
\]
and we can call such a product an \emph{unordered $l$-tuple}. For example if $P$ has $5$ roots, then all the unordered $3$-tuples are 
\[
r_1r_2r_3,r_1r_2r_4,r_1r_2r_5,r_2r_3r_4,r_2r_3r_5,r_3r_4r_5,r_1r_3r_4,r_1r_3r_5,r_1r_4r_5,r_2r_4r_5.
\]
Further more we can say \emph{$l$-tuple with $j$} as a unordered $l$-tuple with index $j$, for example if $P$ has $5$ roots as above we have all $3$-tuples with $2$ are
\[
r_1r_2r_3,r_1r_2r_4,r_1r_2r_5,r_2r_3r_4,r_2r_3r_5,r_2r_4r_5.
\]

Now we can define:
\[
T^P_{l,j}=\sum_{a\in\{\text{All l-tuple with j}\}}a,
\]
where $0$-tuples are defined to be the coefficient of the highest term of polynomial $P$. Later we can just write the sum without explicitly writing down the terms and there is no confusion.
\[
T^P_{l,j}=\sum_{\text{All l-tuple with j}}.
\]
For example,when  $P$ has $5$ roots, then:
\[
T^P_{3,2}=r_1r_2r_3+r_1r_2r_4+r_1r_2r_5+r_2r_3r_4+r_2r_3r_5+r_2r_4r_5.
\] 

The following equation is also clear:
\[
\sum_{\text{All l-tuple without j}}=\sum_{a\in\{\text{All unordered l-tuples}\}}a-\sum_{a\in\{\text{All l-tuple with j}\}}a.
\]

We now have the following definition:
\begin{defn}
 For any polynomial $P\in\mathbb{C}[x]$.	The functions $S^{P}_{l},l\in\{0,1,\dots,m-1\}$ constructed as above are called generalized trigonometric functions associated with polynomial $P$. Namely:
 
 \[
 S^{P}_l(x)=\sum_{j=1}^{m}T^{P}_{l,j}e^{-i r_j x},
 \]
 where we have:
 \[
 T^P_{l,j}=\sum_{\text{All l-tuple with j}}.
 \]
\end{defn}
\section{Properties of generalized trigonometric functions}

Throughout this section $P=\sum_{k=0}^{m}a_k x^k$ is a fixed complex polynomial with roots (written with multiplicity) $r_1,\dots, r_m$. We can consider $a_m=1$ by rescaling the polynomial $P$. Such rescaling will not change the roots. The first result tells us about the Taylor expansion of functions $S^{P}_{l}$.
\begin{thm}
	$S^P_l(x)=\sum_{k=0}^{\infty} b_kx^k$, for $x\in\mathbb{C}$ and $k!b_k\in\mathbb{Z}[i,a_0,a_2,...,a_m]$. 
\end{thm}

\begin{proof}
	We can use the power series for function 
	\[
	e^{-ir_j x}=\sum_{k=0}^{\infty}\frac{(-ir_jx)^k}{k!}=\sum_{k=0}^{\infty}\frac{(-ir_j)^k}{k!}x^k
	\]
	
	then we have
	\begin{eqnarray*}
		S^P_l(x)&=&\sum_{j=1}^{m}T^P_{l,j} e^{-ir_jx}\\ \\
		&=&\sum_{k=0}^{\infty}\sum_{j=1}^{m} \frac{(-i)^k}{k!}T^P_{l,j}r^k_jx^k
	\end{eqnarray*}
	so $b_k$ in the statement of theorem can be computed
	
	\[
	b_k=\sum_{j=1}^{m} \frac{(-i)^k}{k!}T^P_{l,j}r^k_j.
	\]
	
	The sum over $j$ of $T^P_{l,j}r^k_j$ is symmetric among all the roots $r_1,r_2,...,r_m$ therefore $k! b_k$ can be expressed as a integer coefficient polynomial with variables $i, a_0,\dots, a_m$.
\end{proof}

Just as $\sin^2(x)+\cos^2(x)=1$. The generalized trigonometric functions also have such algebraic identities. 

\begin{thm}\label{GA}
	There is a polynomial $F\in\mathbb{C}[x_1,x_2,...,x_m]$ such that all the coefficients of $F$ are algebraic over $a_0,...,a_m$ and $F(S^P_0(x),...,S^P_{m-1}(x))=0$ for all $x\in\mathbb{C}$.
\end{thm}

To prove this we need to work on linear transformations of $S^P_l(x)$, that is, given $m$ complex numbers $L_j,j=0,1...m-1$ we can define a function 
\[
L(x)=\sum_{j=0}^{m-1}L_jS^P_j(x),
\]
which is a linear combination of all functions $S^P_l$. Considering such linear combinations is very helpful, and indeed for a non degenerate matrix $L_{i,j},i,j=0,...,{m-1}$ the following functions: 

\[LS_i(x)=\sum_{j=0}^{m-1}L_{i,j}S^P_j(x),i=0...m-1\]
are in some sense equivalent to $S^P_l(x)$.

\begin{proof}[Proof of theorem \ref{GA}]
	As above, given $m$ complex numbers $L_j,j=0,1,...,m-1$ we can define a function 
	\[
	f_0(x)=L(x)=\sum_{j=0}^{m-1}L_jS^P_j(x),
	\]
	then for any $l=1,2,...,m-1$ we define
	\[
	f_l(x)=f'_{l-1}(x).
	\]
	
	Now if $\lambda f_0= f'_{m-1}$ for some number $\lambda\neq 0$ we can construct a matrix
	
	\[
	M=
	\begin{bmatrix}
	f_0 & f_1 & f_2 & \dots & f_{m-2} & f_{m-1}\\
	f_1 & f_2 & f_3 & \dots & f_{m-1} & \lambda f_0\\
	f_2 & f_3 & f_4 & \dots & \lambda f_{0} & \lambda f_1\\
	\vdots & \vdots & \vdots & \ddots  & \vdots & \vdots\\
	f_{m-1} & \lambda f_{0} & \lambda f_{1} & \dots &  \lambda f_{m-3} & \lambda f_{m-2}\\
	\end{bmatrix}
	\]
	
	The above matrix is actually a function on variable $x$. Define $D(x)=\det(M)$ which is a smooth function of $x$, then applying Jacobi formula we get:
	\[
	D'(x)=\tr(\adj(M)M'(x)).
	\]
	Here $\tr, \adj$ stand for the trace and adjugate of matrices. 
	Now we see that: 
	\begin{eqnarray*}
	M'&=&
	\begin{bmatrix}
	f'_0 & f'_1 & f'_2 & \dots & f'_{m-2} & f'_{m-1}\\
	f'_1 & f'_2 & f'_3 & \dots & f'_{m-1} & \lambda f'_0\\
	f'_2 & f'_3 & f'_4 & \dots & \lambda f'_{0} & \lambda f'_1\\
	\vdots & \vdots & \vdots & \ddots  & \vdots & \vdots\\
	f'_{m-1} & \lambda f'_{0} & \lambda f'_{1} & \dots &  \lambda f'_{m-3} & \lambda f'_{m-2}\\
	\end{bmatrix}\\
	&=&
	\begin{bmatrix}
	f_1 & f_2 & f_3 & \dots & f_{m-1} & \lambda f_0\\
	f_2 & f_3 & f_4 & \dots & \lambda f_0 & \lambda f_1\\
	f_3 & f_4 & f_5 & \dots & \lambda f_1 & \lambda f_2\\
	\vdots & \vdots & \vdots & \ddots  & \vdots & \vdots\\
	\lambda f_0 & \lambda f_1 & \lambda f_2 & \dots &  \lambda f_{m-2} & \lambda f_{m-1}\\
	\end{bmatrix}	
	\end{eqnarray*}
	
	Now we have the formula:
	\[
	\adj(M) M=(\det M )I,
	\]
	where $I$ is the identity matrix. And observe that the column vectors of $M'$ is a permutation of column vectors of $M$ with one column multiplied by a number $\lambda$. It is then easy to see that $\adj(M) M'$ is an off diagonal matrix. Therefore we see that:
	\[
	D'(x)=\tr(\adj(M)M'(x))=0.
	\]
	This implies that for any $x\in\mathbb{C}$
	\[
	\det M=\det M(0),
	\]
	and the value $\det M(0)$ is algebraic over $S^P_l(0),l=0...m-1$ which are easily shown to be algebraic over the coefficients $a_k,k=0,...,m$ of polynomial $P$. So it is enough to show that we can choose $L_j,j=0...m-1$ and $\lambda$ algebraic over $a_k$ such that the following relations do hold
	
	\[
	f_l(x)=f'_{l-1}(x),l=1,2...m-1
	\]
	
	\[
	f_0(x)=\lambda^{-1} f'_{m-1}(x).
	\]
	
	Now we study the derivatives of functions $S^P_l(x)$,
	\[
	{S^P_l}'(x)=\sum_{j=1}^{m}T^P_{l,j}(-ir_j) e^{-ir_jx},
	\]
	and the value of $T^P_{l,j}r_j$ can be easily computed as follows. For $l=0$ we have
	\[
	T^P_{0,j}=a_m=1,
	\]
	and so
	\[
	T^P_{0,j}r_j=r_j.
	\]
	For $l=1,2,...,m-2$ we have:
	
	\begin{eqnarray*}
		T^P_{l,j}r_j&=&r_j\sum_{\text{l-tuples with j}}\\ \\
		&=&r_j\left(\sum_{\text{all l-tuples}}-\sum_{\text{l-tuples without j}}\right)\\ \\
		&=&r_j\left((-1)^l a_{m-l}-\sum_{\text{l-tuples without j}}\right)\\ \\
		&=&r_j(-1)^l a_{m-l}-\sum_{\text{l+1-tuples with j}}
	\end{eqnarray*}
	where we have used the fact that the following expression:
	\[
	(-1)^l\sum_{\text{all l-tuples}}
	\]
	with roots of the monic polynomial $P(n)$ gives the coefficient of the term $n^{m-l}$ and the fact that for $l\leq m-2$
	
	\[
	r_j\sum_{\text{l-tuples without j}}=\sum_{\text{l+1-tuples with j}}.
	\]
	Next we consider the case when $l=m-1$. We have
	\begin{eqnarray*}
		T^P_{m-1,j}r_j&=&r_j\sum_{\text{(m-1)-tuples with j}}\\ \\
		&=&r_j\left(\sum_{\text{all m-1-tuples}}-\sum_{\text{m-1-tuples without j}}\right)\\ \\
		&=&r_j\left((-1)^{m-1} a_1-\sum_{\text{(m-1)-tuples without j}}\right)\\ \\
		&=&r_j(-1)^{m-1} a_{1}-\sum_{\text{m-tuples with j}}\\ \\
		&=&r_j(-1)^{m-1} a_{1}-(-1)^m a_0.
	\end{eqnarray*}
	In all we have:
	\[
	T^P_{l,j}r_j = \begin{cases}
	r_j, & \text{for } l=0\\
	r_j(-1)^l a_{m-l}-\sum_{\text{l+1-tuples with j}}, & \text{for } 0<l<m-1\\
	r_j(-1)^{m-1} a_{1}-(-1)^m a_0, & \text{for } l=m-1
	\end{cases}
	\]
	and therefore we have
	\begin{eqnarray*}
		{S^P_l}'(x)&=&\sum_{j=1}^{m}T^P_{l,j}(-ir_j) e^{-ir_jx}\\ \\
		&=&\sum_{j=1}^{m}e^{-ir_jx}
		\begin{cases}
			r_j, & \text{for } l=0\\
			r_j(-1)^l a_{m-l}-\sum_{\text{l+1-tuples with j}}, & \text{for } 0<l<m-1\\
			r_j(-1)^{m-1} a_{1}-(-1)^m a_0, & \text{for } l=m-1
		\end{cases}\\ \\
		&=&
		\begin{cases}
			-iS^P_1(x), & \text{for } l=0\\
			(-1)^{l+1}ia_{m-l}S^P_1(x)+iS^P_{l+1}(x), & \text{for } 0<l<m-1\\
			(-1)^{m}ia_1S^P_1(x)+(-1)^mia_0S^P_0(x), & \text{for } l=m-1
		\end{cases}\\ \\
	\end{eqnarray*}
	which can be expressed in the following matrix form:
	\[
	\begin{bmatrix}
	{S^P_0}' \\
	{S^P_1}' \\
	{S^P_2}' \\
	\vdots \\
	{S^P_{m-1}}' \\
	\end{bmatrix}
	=
	\begin{bmatrix}
	0 & -i & 0 & 0 & \dots & 0 \\
	0 & ia_{m-1} & i & 0 & \dots & 0 \\
	0 & -ia_{m-2} & 0 & i & \dots & 0 \\
	\vdots & \vdots & \vdots & \vdots & \dots & \vdots \\
	0 & (-1)^{m-1}ia_{2} & 0 & 0 & \dots & i \\
	(-1)^mia_0 & (-1)^mia_{1} & 0 & 0 & \dots & 0 \\
	\end{bmatrix}
	\begin{bmatrix}
	S^P_0 \\
	S^P_1 \\
	S^P_2 \\
	\vdots \\
	S^P_{m-1} \\
	\end{bmatrix}
	\]
	We denote the matrix in this equation as $K$. Then we have
	\begin{eqnarray*}
		f_1=f'_0&=&\sum_{j=0}^{m-1}L_j{S^P_j}'(x)\\ \\
		&=&
		\begin{bmatrix}
			L_0 & L_1 & L_2 & \dots & L_{m-1}
		\end{bmatrix}
		K
		\begin{bmatrix}
			S^P_0 \\
			S^P_1 \\
			S^P_2 \\
			\vdots \\
			S^P_{m-1} \\
		\end{bmatrix}
	\end{eqnarray*}
	We can do the above step to find all $f_2,f_3...f_{m-1}$ and the last step will be computing the derivative of $f_{m-1}$. We have
	\[
	f'_{m-1}=
	\begin{bmatrix}
	L_0 & L_1 & L_2 & \dots & L_{m-1}
	\end{bmatrix}
	K^m
	\begin{bmatrix}
	S^P_0 \\
	S^P_1 \\
	S^P_2 \\
	\vdots \\
	S^P_{m-1} \\
	\end{bmatrix}
	\]
	
	and we want the following relation:
	\[
	f'_{m-1}=\lambda^{-1} f_0,
	\]
	which will be satisfied if
	\[
	\begin{bmatrix}
	L_0 & L_1 & L_2 & \dots & L_{m-1}
	\end{bmatrix}
	M^m
	=\lambda
	\begin{bmatrix}
	L_0 & L_1 & L_2 & \dots & L_{m-1}
	\end{bmatrix}
	\]
	this will be satisfied if $\lambda^{-1}$ is an eigenvalue of $M=K^m$. Such $\lambda^{-1}$ exists and is algebraic over the entries of matrix $K$ and therefore also algebraic over $a_0,a_1,...,a_{m-1}$. Whenever $\lambda^{-1}$ is an eigenvalue of $K^m$, we can find non trivial vector 
	
	\[
	\begin{bmatrix}
	L_0 & L_1 & L_2 & \dots & L_{m-1}
	\end{bmatrix}
	\]
	where all the numbers $L_i's$ can be chosen to be algebraic over $a_0,a_1,...,a_{m-1}$, and then the fact mentioned earlier in this proof 
	\[
	\det M=\det M(0)
	\]
	determines an algebraic relation satisfying all the conditions in the statement of this theorem.
\end{proof}
\section{Examples of generalized trigonometric functions}\label{SC}
Throughout this section $m$ is a fixed integer and $P(x)=x^m-1$. Then we write:
\[
\zeta=e^{\frac{2\pi i}{m}}, \eta=\zeta^2=e^{\frac{2\pi i}{m}}.
\]
We can construct functions $S^{P}_l,l\in\{0,1,\dots,m-1\}$. Then we can further obtain the following functions:
\[
S_l(x)=\frac{1}{m}S^{P}_l(i\eta x).
\]
The reason for rescaling will become clear later. We can explicitly write down the functions $S_l$:

	\[
	S_l(x)=\frac{1}{m\eta^l}\sum_{j=0}^{m-1}\zeta^{lj}e^{\eta\zeta^{j}x}.
	\] 

\begin{example}
	If m=1 then $\zeta=1,\eta=-1$ and $S_0=e^{-x}$.
\end{example}

\begin{example}
	If m=2 then $\zeta=-1,\eta=i$ and $S_0=\cos(x),S_1=\sin(x)$
\end{example}

The above two examples justify the rescaling and the name "generalized trigonometric functions" . We can study the Taylor series and algebraic identities as a special case of the general result discussed in the previous section. We will state the results without repeating the proofs.

\begin{thm}[Power series]
	If $l\neq 0$ we have
	\[
	S_l=\frac{1}{\zeta^l}\sum_{k=1}^{\infty}\frac{1}{(km-l)!}x^{km-l}(-1)^k.
	\]
	If $l=0$ we have
	\[
	S_l=\frac{1}{\zeta^l}\sum_{k=0}^{\infty}\frac{1}{(km)!}x^{km}(-1)^k,
	\]
	where we consider $x\in\mathbb{C}$.
\end{thm}

\begin{thm}[Algebraic identities]\label{AIG}
	For any $m$ being integers larger than $1$, the functions $S_l(x)_{l=0,1,2...,m-1}$ are not algebraically independent in the sense that we can find a non trivial polynomial $F$ in $\mathbb{Z}[x_1,...,x_m]$ of order $m$ such that:
	
	$$F(S_0(x),...,S_{m-1}(x))=0$$ for all $x\in\mathbb{R}$ and therefore for any $x\in\mathbb{C}$.
\end{thm}

In fact by going through proof of theorem \ref{GA} we can actually obtain explicitly the polynomial $F$. Indeed, if we set:
\[
f_l(x)=\zeta^{l} S_l(x),
\]  
and construct the matrix function:
	\[
M(x)=
\begin{bmatrix}
f_0 & f_1 & f_2 & \dots & f_{m-2} & f_{m-1}\\
f_1 & f_2 & f_3 & \dots & f_{m-1} & -f_0\\
f_2 & f_3 & f_4 & \dots & -f_{0} & -f_1\\
\vdots & \vdots & \vdots & \ddots  & \vdots & \vdots\\
f_{m-1} & -f_{0} & -f_{1} & \dots &  -f_{m-3} & -f_{m-2}\\
\end{bmatrix}.
\]

Then the algebraic identity satisfied by functions $S_l$ is:
\[
\det M(x)=(-1)^{m-1}.
\]

\begin{example}
	When $m=2$ we see that:
	\[
	M=
	\begin{bmatrix}
	f_0 & f_1 \\ \\
	f_1 & -f_0 
	\end{bmatrix}
	\]
	and thus we have
	\[
	\det M=-f_0^2-f_1^2=-1,
	\]
	and this is equivalent to 
	\[
	S_0^2+S_1^2=1,
	\]
	which is just the familiar identity:
	\[
	\sin^2(x)+\cos^2(x)=1.
	\]
\end{example}

\begin{example}
	Similarly we can get the corresponding equations for other $m$ as well 
	\[
	m=3:
	-S_0^3(x)+S_1^{3}(x)-S_2^{3}(x)-3S_0(x)S_1(x)S_2(x)=1
	\]
	
	But for large $m$ the polynomial is very complicated for example when $m=7$ the corresponding polynomial will contain $246$ monomials. If we expand the functions $S_j$ into power series and compare the coefficients of each $x^n$ we can show the following result:
	
	\begin{cor}
		For any integer such that $3|n$, we can decompose $n$ into sum of three non-negative integers $n=k_1+k_2+k_3$. Then we have the following relation:
		
		\[
		\sum_{A}\frac{1}{k_1!k_2!k_3!}=3\sum_{B}\frac{1}{k_1!k_2!k_3!}
		\]
		
		where in $\sum_{A}$ we have $k_1,k_2,k_3$ are congruent modulo $3$ and in $\sum_{B}$ we have $k_1,k_2,k_3$ in three different classes modulo $3$. 
	\end{cor}
\end{example}

Next recall that it is well known that for $m=2$:

\[
\sin(x_1+x_2)=\sin(x_1)\cos(x_2)+\sin(x_2)\cos(x_1).
\]
It is then very natural to think if general relations hold for other $m$ as well, and this is indeed the case.

\begin{thm}[Summation formula]
	For any positive integer $m$, and any integer $l$ between $0$ and $m-1$, $ S_l(x_1+x_2)$ can be written as linear combination of $S_{l_1}(x_1)S_{l_2}(x_2)$ where $l_1,l_2$ also range over $0$ to $m-1$.
\end{thm}
\begin{rem}
	We can also obtain summation formula for general $S^{P}_l$ functions with the same method.
\end{rem}
\begin{proof}
	Recall by definition:
	
	\[
	S_l(x)_{l=0,1,...,m-1}=\frac{1}{m\eta^l}\sum_{j=0}^{m-1}\zeta^{lj}e^{\eta\zeta^{j}x},
	\]
	where $\zeta=e^{\frac{2\pi i}{m}}$ and $\eta=e^{\frac{\pi i}{m}}$. Then observe that:
	
	\begin{eqnarray*}
		\sum_{k=0}^{k=m-1}\frac{1}{m}\zeta^{rk}S_l(x_1+\zeta^kx_2) &=& \sum_{k=0}^{m-1}\frac{1}{m^2}\zeta^{rk}\sum_{j=0}^{m-1}\frac{1}{\eta^l}\zeta^{lj}e^{\eta\zeta^jx_1}e^{\eta\zeta^{j+k}x_2}\\ \\
		&=&\sum_{j=0}^{m-1}\sum_{k=0}^{m-1}\frac{1}{m^2}\zeta^{rk}\frac{1}{\eta^l}\zeta^{lj}e^{\eta\zeta^jx}e^{\eta\zeta^{j+k}x_2}\\ \\
		&=&\frac{1}{m^2\eta^l}\sum_{j=0}^{m-1}\zeta^{lj+r(m-j)}e^{\eta\zeta^{j}x_1}\sum_{k=0}^{m-1}\zeta^{rk-r(m-j)}e^{\eta\zeta^{j+k}x_2}.
	\end{eqnarray*}
	In the sum of $k$ from $0$ to $m-1$ we can change the sum of $k$ from $-j$ to $-j+m-1$ and the result is unchanged as the terms in the sum is periodic with period $m$. The result of the sum of $k$ is then:
	
	\begin{eqnarray*}
		\sum_{k=0}^{m-1}\zeta^{rk-r(m-j)}e^{\eta\zeta^{j+k}x_2}&=&\sum_{k=0}^{m-1}\zeta^{r(k-j)-r(m-j)}e^{\eta\zeta^{j+(k-j)}x_2}\\ \\
		&=&\sum_{k=0}^{m-1}\zeta^{rk}e^{\eta\zeta^{k}x_2}\\ \\
		&=&m\eta^kS_{k}(x_2).
	\end{eqnarray*}
	
	Then we can go on computing the sum:
	\[
	\sum_{j=0}^{m-1}\zeta^{lj+r(m-j)}e^{\eta\zeta^{j}x_1}\sum_{k=0}^{m-1}\zeta^{rk-r(m-j)}e^{\eta\zeta^{j+k}x_2}=m^2\eta^{k+(l-r)_m}S_{(l-r)_m}(x_1)S_k(x_2),
	\]
	where the notation $(n)_m$ is the smallest non negative number of the form $n+km$. 
	
	In all we have:
	
	\[
	\sum_{k=0}^{k=m-1}\frac{1}{m}\zeta^{rk}S_l(x_1+\zeta^kx_2)=\eta^{k+(-r)_m}S_{(l-r)_m}(x_1)S_{k}(x_2).
	\]
	The above relation holds for any $r$, and for $r=0,1,2,...,m-1$ we got $m-1$ linear independent equations because the coefficient matrix is actually a Vandermonde matrix and it is easily seen to be non-singular. We can solve the equation and write $S_l(x_1+x_2)$ as a linear combination of $S_l(x_1)S_{l'}(x_2)$ with $l,l'$ ranging over $0,1,2...m-1$.
\end{proof}

\section{An application of generalized trigonometric functions: Evaluating series of rational function}
Let $f(n)=\frac{Q(n)}{P(n)}$ be a rational function where $P,Q$ are polynomials with complex coefficients. Then we can study the sum  
\[
\sum_{n=0}^{\infty}f(n)
\] 
We know the series converges when $\deg Q<\deg P-1$ and in this case we can use partial decomposition to write
\[
f(n)=\sum_{k=1}^{m}\frac{b_k}{(n+r_k)^{s_k}}
\]
for suitable integers $m,s_k$ and numbers $b_k,r_k$.
Then we can use polygamma function $\psi^{(m)}(z)$ which is defined to be the $m$-th derivative of the gamma function:
\[
\psi^{(m)}(z)=\frac{d^{m+1}}{dz^{m+1}}\ln\Gamma(z)
\]
and it is well known that:
\[
\psi^{(m)}(z)=(-1)^{m+1}m!\sum_{n=0}^{\infty}\frac{1}{(z+n)^{m+1}}
\]
therefore we have
\[
\sum_{n=0}^{\infty}\frac{1}{(n+z)^{m+1}}=(-1)^{m+1}\frac{\psi^{(m)}(z)}{m!}.
\]

In this way we can express any infinite sum of rational functions by linear compositions of values of polygamma functions. Further, if the rational function $f(n)$ has a partial fractional decomposition mentioned above with the coefficients $s_k=1$ and $r_k$ are rational then from the results in \cite{LH} we get the following results without polygamma functions involved:

\begin{eqnarray*}
	\sum_{n=0}^{\infty}\frac{1}{f(n)}&=&\sum_{n=0}^{\infty}\sum_{k=1}^{m}\frac{b_k}{(n+r_k)}=\sum_{n=0}^{\infty}\sum_{k=1}^{m}\frac{c_k}{(q_kn+p_k)} \\ \\
	&=&\sum_{k=1}^{m}\sum_{j=1}^{q_k-1} \frac{c_k}{q_k}(1-e^{\frac{-2\pi i j p_k}{k}})\log (1-e^{\frac{-2\pi i j }{k}}).
\end{eqnarray*}
This formula is very useful when we want to determine whether a infinite sum of rational function is transcendental by using the Baker's logarithmic linear form see \cite{ADHIKARI20011} and \cite{Pilerhood2007}.Note that in \cite{Murty2010} the same result was used to show that 'almost all' the generalized Euler constants are transcendental. 

For general rational functions we do not have such nice result but if we consider the two sided sum:

\[
\lim_{k\to\infty}\sum_{n=-k}^{k}f(n),
\]
then it is possible by using residue theorem of the following integral and letting $N\to\infty$:

\[
I(N)=\int_{R_N} \frac{\pi\cot (\pi z)}{Q(z)}dz,
\]
where $R_N$ is the rectangle with vertex $(N+0.5)(\pm 1\pm i)$, the result will be some liner combination of $\cot (\pi r_k)$ for $r_k$ being the roots of polynomial $Q$ over $\mathbb{C}$. A special case when $f(n)=n^{2k}-1$ was considered in \cite{Bundschuh1979} to show that:

\[
\sum_{|n|\geq 2}\frac{1}{n^{2k}-1}=2-\frac{\delta}{2s}-\pi i\sum_{\sigma=1,\sigma\neq s/2}^{s-1}r^\sigma\frac{e^{2\pi i r^{\sigma}}+1}{e^{2\pi i r^\sigma}-1}
\]

where $r=e^{\frac{\pi i }{s}}$ and $\delta=1$ for odd $s$ and $\delta=2$ for even $s$.

For some easy polynomials for example $f(n)=n^2$, there are several different proofs of the following fact which can be found in \cite{Chapman2003}

\[
\sum_{n=1}^{\infty}\frac{1}{n^2}=\frac{\pi^2}{6}.
\]

There are two 'standard textbook proofs', the one using the residue theorem mentioned above and the other one using Dirichlet's theorem of Fourier series. The residue method can be easily altered to compute the series of general rational function. Compare to this, finding a Fourier series argument is more difficult but such a method can lead us to a more detailed understanding of series of rational functions. Here we present a Fourier series method of computing series of rational functions as an application of later defined generalized trigonometric functions.

It turns out that the following functions will be useful:

\begin{eqnarray*}
	R^P_l(x)=\sum_{j=1}^{m}\frac{1}{e^{-ir_j\pi}-e^{ir_j\pi}}T^P_{l,j} e^{-ir_jx}
\end{eqnarray*}

It is easy to compute the Fourier series:

\begin{eqnarray*}
	c^P_{-n,l}&=&\frac{1}{2\pi}\int_{-\pi}^{\pi}R^P_l(x)e^{inx}dx \\ \\
	&=&\frac{(-1)^n}{2\pi  i}\sum_{j=1}^{m}T^P_{l,j}\frac{1}{n-r_j}\\ \\
	&=&\frac{(-1)^n}{2\pi  i P(n)}\sum_{k=0}^{m-2}\sum_{j=1}^{m}T^p_{l,j}(-1)^{m-k-1}\sum_{\text{m-k-1-tuples without j}}n^k\\ &+&\frac{(-1)^n}{2\pi  i P(n)}\sum_{j=1}^{m}T^P_{l,j}n^{m-1}\\ \\
	&=&\frac{(-1)^n}{2\pi  i P(n)}\sum_{k=0}^{m-2}\sum_{j=1}^{m}T^p_{l,j}(-1)^{m-k-1}(\sum_{\text{m-k-1-tuples}}-\sum_{\text{m-k-1-tuples with j}})n^k\\ &+&\frac{(-1)^n}{2\pi  i P(n)}\sum_{j=1}^{m}T^P_{l,j}n^{m-1}\\ \\
	&=&\frac{(-1)^n}{2\pi  i P(n)}\sum_{k=0}^{m-2}\sum_{j=1}^{m}T^p_{l,j}(-1)^{m-k-1}((-1)^{m-k-1}a_{k+1}-T^p_{m-k-1,j})n^k\\ &+&\frac{(-1)^n}{2\pi  i P(n)}\sum_{j=1}^{m}T^P_{l,j}n^{m-1}\\ \\
	&=&\sum_{k=0}^{m-1}(-1)^n\frac{C_{l,k}n^k}{ P(n)}
\end{eqnarray*}
where $C_{l,k}$ are suitable coefficients, and together they form a $m\times (m+1)$ matrix. We denote this matrix $C(P)$ as the $m\times m$ matrix by taking the first $m$ rows, namely, $C_{l,k}$ with $l,k=0,...,k-1$.  It depends only on polynomial $P$.

\begin{defn}
	For any polynomial $P\in\mathbb{C}[x]$ the matrix $C(P)$ obtained as above is called the associated matrix of $P$.
\end{defn}

Now we can use Dirichlet's theorem to obtain:
\begin{eqnarray*}
	\frac{R^P_l(\pi)+R^P_l(-\pi)}{2}&=&\sum_{n\in\mathbb{Z}}c^P_{-n,l}\\ \\
	&=&\sum_{n\in\mathbb{Z}}\sum_{k=0}^{m-1}\frac{C_{l,k}n^k}{ P(n)}\\ \\
	&=&\sum_{k=0}^{m-1}C_{l,k}\sum_{n\in\mathbb{Z}}\frac{n^k}{P(n)}
\end{eqnarray*}

As $l$ ranges over $0,1,2...m-1$ we can get $m$ linear equations of $m$ numbers

\[
\sum_{n\in\mathbb{Z}}\frac{n^k}{P(n)},k=0,1,..,m-1
\]
and similarly:

\begin{eqnarray*}
	R^P_l(0)&=&\sum_{n\in\mathbb{Z}}c^P_{-n,l}(-1)^n\\ \\
	&=&\sum_{n\in\mathbb{Z}}\sum_{k=0}^{m-1}(-1)^n\frac{C_{l,k}n^k}{ P(n)}\\ \\
	&=&\sum_{k=0}^{m-1}C_{l,k}\sum_{n\in\mathbb{Z}}(-1)^n\frac{n^k}{P(n)}.
\end{eqnarray*}
We can get $m$ linear equations with $m$ numbers:

\[
\sum_{n\in\mathbb{Z}}(-1)^n\frac{n^k}{P(n)},k=0,1,...,m-1.
\]
Therefore we get the following conclusion:

\begin{thm}
	Let $P\in\mathbb{C}[x]$. If the associated matrix $C(P)$ is non degenerate and if the following series converges:
	\[
	\sum_{n\in\mathbb{Z}}\frac{n^k}{P(n)},\sum_{n\in\mathbb{Z}}(-1)^n\frac{n^k}{P(n)},k=0,1,..,m-1
	\]
	then they are in the number field:
	\[
	\mathbb{Q}(i,\pi,\mathcal{R},\mathcal{E})
	\]
	where $\mathcal{R}$ is the set of roots of polynomial $P$ and $\mathcal{E}$ is the set consisting of $e^{-ir\pi}$ for $r$ the root of $P$. And we can find the expression explicitly.
\end{thm}

\section{Examples of evaluating series of rations functions}
We first consider a special case when $P(x)=x^3+x^2+1$. Then we can write down the generalized trigonometric functions associated with $P$ as follows,

\[
R^P_l(x)=\sum_{j=1}^{3}\frac{1}{e^{-ir_j\pi}-e^{ir_j\pi}}T^P_{l,j} e^{-ir_jx}
\]
where $r_1,r_2,r_3$ are the roots of $x^3+x^2+1=0$. We can compute the $T^P_{l,j}$ explicitly:

\begin{eqnarray*}
	T^P_{0,j}&=&1\\ \\
	T^P_{1,j}&=&r_j\\ \\
	T^P_{2,j}&=&-r_j-r_j^2
\end{eqnarray*}
and now we can consider the following series:

\[
A_k=\sum_{-\infty}^{\infty}\frac{n^k}{n^3+n^2+1},B_k=\sum_{-\infty}^{\infty}(-1)^n\frac{n^k}{n^3+n^2+1}
\]
for $k=0,1,2$. The matrix $C(P)$ is
\[
C(P)=\frac{1}{2\pi i}
\begin{bmatrix}
3 & 2 & 0 \\
-1 & 0 & -3 \\
0 & 3 & 2 \\
\end{bmatrix}
\]
and we have

\[
C(P)
\begin{bmatrix}
B_2\\
B_1\\
B_0
\end{bmatrix}
=
\begin{bmatrix}
R^P_0(0)\\
R^P_1(0)\\
R^P_2(0)
\end{bmatrix}
\]
therefore we obtain:

\[
\begin{bmatrix}
B_2\\
B_1\\
B_0
\end{bmatrix}
=
C(P)^{-1}
\begin{bmatrix}
R^P_0(0)\\
R^P_1(0)\\
R^P_2(0)
\end{bmatrix}
\]
This is an explicit formula for series

\begin{eqnarray*}
	B_k=\sum_{-\infty}^{\infty}(-1)^n\frac{n^k}{n^3+n^2+1}
\end{eqnarray*}

\begin{eqnarray*}
	\sum_{-\infty}^{\infty}(-1)^n\frac{n^2}{n^3+n^2+1}&=&
	\frac{2\pi i}{31}\sum_{j=1}^{3}\frac{9-4r_j-6r_j^{-1}}{e^{-ir_j\pi}-e^{ir_j\pi}} \\ \\
	\sum_{-\infty}^{\infty}(-1)^n\frac{n}{n^3+n^2+1}&=&
	\frac{2\pi i}{31}\sum_{j=1}^{3}\frac{2+6r_j+9r_j^{-1}}{e^{-ir_j\pi}-e^{ir_j\pi}}\\ \\
	\sum_{-\infty}^{\infty}(-1)^n\frac{1}{n^3+n^2+1}&=&
	\frac{2\pi i}{31}\sum_{j=1}^{3}\frac{-3-9r_j+2r_j^{-1}}{e^{-ir_j\pi}-e^{ir_j\pi}}
\end{eqnarray*}
and similarly we can get also formulas for 

\[
A_k=\sum_{-\infty}^{\infty}\frac{n^k}{n^3+n^2+1}
\]

\begin{eqnarray*}
	\sum_{-\infty}^{\infty}\frac{n^2}{n^3+n^2+1}&=&
	\frac{\pi i}{31}\sum_{j=1}^{3}(9-4r_j-6r_j^{-1})\frac{e^{-ir_j\pi}+e^{ir_j\pi}}{e^{-ir_j\pi}-e^{ir_j\pi}} \\ \\
	\sum_{-\infty}^{\infty}\frac{n}{n^3+n^2+1}&=&
	\frac{\pi i}{31}\sum_{j=1}^{3}(2+6r_j+9r_j^{-1})\frac{e^{-ir_j\pi}+e^{ir_j\pi}}{e^{-ir_j\pi}-e^{ir_j\pi}}\\ \\
	\sum_{-\infty}^{\infty}\frac{1}{n^3+n^2+1}&=&
	\frac{\pi i}{31}\sum_{j=1}^{3}(-3-9r_j+2r_j^{-1})\frac{e^{-ir_j\pi}+e^{ir_j\pi}}{e^{-ir_j\pi}-e^{ir_j\pi}}
\end{eqnarray*}

We discussed section \ref{SC} a special class of rescaled generalized trigonometric functions $S_l$ associated with $P(x)=x^m-1$. Just as the general case, we can compute the Fourier coefficients of $S_l$ and use Dirichlet's theorem. In order to state the result, we make some temporary notations:
\begin{itemize}
	\item [1]: $
	\Delta_l=S_l(\pi)-S_l(-\pi), l\in\{0,1,\dots,m-1\}.
	$
	\item[2]: $(n)_m$ is the smallest non negative number of the form $n+km$. 
	\item[3]: $J_k^{l}=\eta^{m-1-k-l+(m-1-k+l)_m}\Delta_{(m-1-k+l)_m},k,l\in\{0,1,\dots,m-1\}$.
	\item[4]: For all natural number $n$, $M_l(n)=
	\begin{cases}
	-2i^ln^l & \text{ if $l$ is even}\\
	-2i^{m+l}n^{m+l}& \text{ if $l$ is odd}
	\end{cases}
	$
\end{itemize}

Then we have the following results:
\[
S_l(0)=-\frac{1}{2\pi}J^l_0+\frac{1}{2\pi}\sum_{n\geq 1}(-1)^n\sum_{k=0}^{m-1}(-1)^kJ^l_k\frac{M_k}{n^{2m}+1},
\]
\[
\frac{S_l(\pi)+S_l(-\pi)}{2}=-\frac{1}{2\pi}J^l_0+\frac{1}{2\pi}\sum_{n\geq 1}\sum_{k=0}^{m-1}(-1)^kJ^l_k\frac{M_k}{n^{2m}+1}.
\]
Then assume that we can solve the above $2m$ linear equations we can find the value of all the sums of form $$\sum_{n\geq 1}\frac{n^{2k}}{n^{2m}+1}(\pm 1)^n.$$

Now our problem is to show that we can indeed solve the equations, namely, the non-singularity of the following matrix:
\[
A_m[lk]=(-1)^{k+1}J^l_k(i)^{k+m(k)_2}.
\]
Here we indexed the matrix from $0$ instead of the usual convention $1$. Now it is enough to conclude the computation by showing that matrix $A_m$ is non degenerate for any positive integer $m$ (not only for odd integers).

\begin{thm}
	Matrix $A_m$ defined as above is not degenerate for any positive integer $m$.
\end{thm}

\begin{proof}
	We want to show that the determinant of $A_m$ defined by following formula is non vanishing:
	\[
	A_m[lk]=(-1)^{k+1}J^l_k(i)^{k+m(k)_2}.
	\]
	We can ignore the coefficients $(-1)^{k+1}(i)^{k+m(k)_2}$ because those will just affect the determinant by multiplying $\pm 1$. We can write down the matrix explicitly:
	\[
	J_k^l=\eta^{m-1-k-l+(m-1-k+l)_m}\Delta_{(m-1-k+l)_m},
	\]
	
	\[
	J=
	\begin{bmatrix}
	\eta^{2m-2}\Delta_{m-1} & \eta^{2m-4}\Delta_{m-2} & \eta^{2m-6}\Delta_{m-3} & \dots &  \eta^{2}\Delta_{1} & \eta^{0}\Delta_{0}\\
	\eta^{m-2}\Delta_{0} & \eta^{2m-4}\Delta_{m-1} & \eta^{2m-6}\Delta_{m-2} & \dots &  \eta^{2}\Delta_{2} & \eta^{0}\Delta_{1}\\
	\eta^{m-2}\Delta_{1} & \eta^{m-4}\Delta_{0} & \eta^{2m-6}\Delta_{m-1} & \dots &  \eta^{2}\Delta_{3} & \eta^{0}\Delta_{2}\\
	\vdots & \vdots & \vdots & \ddots  & \vdots & \vdots\\
	\eta^{m-2}\Delta_{m-2} & \eta^{m-4}\Delta_{m-3} & \eta^{m-6}\Delta_{m-4} & \dots &  \eta^{2-m}\Delta_{0} & \eta^{0}\Delta_{m-1}\\
	\end{bmatrix}
	\]
	
	Then we see that we can factor out some powers of $\eta$ and this will only affect the determinant by some power of $\eta$. This power is $m^2+2+4+6+...+2m$ which is a multiple of $m$. Therefore the power of $\eta$ being factored out is actually $\pm 1$. It is then enough to show that the following matrix is non degenerate:
	
	\[
	\begin{bmatrix}
	-\Delta_{m-1} & -\Delta_{m-2} & -\Delta_{m-3} & \dots &  -\Delta_{1} & \Delta_{0}\\
	\Delta_{0} & -\Delta_{m-1} & -\Delta_{m-2} & \dots &  -\Delta_{2} & \Delta_{1}\\
	\Delta_{1} & \Delta_{0} & -\Delta_{m-1} & \dots &  -\Delta_{3} & \Delta_{2}\\
	\vdots & \vdots & \vdots & \ddots  & \vdots & \vdots\\
	\Delta_{m-2} & \Delta_{m-3} & \Delta_{m-4} & \dots &  \Delta_{0} & \Delta_{m-1}\\
	\end{bmatrix}.
	\]
	
	Now
	$$\Delta_l=\frac{1}{m\eta^l}\sum_{j=0}^{m-1}\zeta^{lj}(e^{\eta\zeta^{j}\pi}-e^{-\eta\zeta^{j}\pi})$$
	
	It is convenient to write the above matrix in terms of $\eta^l\Delta_l$:
	
	\[
	\begin{bmatrix}
	-\eta^{-(m-1)}\eta^{m-1}\Delta_{m-1} & -\eta^{-(m-2)}\eta^{m-2}\Delta_{m-2} & \dots &  -\eta^{-1}\eta^{1}\Delta_{1} & \eta^{0}\eta^{0}\Delta_{0}\\
	\eta^{0}\eta^{0}\Delta_{0} & -\eta^{-(m-1)}\eta^{m-1}\Delta_{m-1} & \dots &  -\eta^{-2}\eta^{2}\Delta_{2} & \eta^{-1}\eta^{1}\Delta_{1}\\
	\eta^{-1}\eta^{1}\Delta_{1} & \eta^{0}\eta^{0}\Delta_{0}& \dots &  -\eta^{-3}\eta^{3}\Delta_{3} & \eta^{-2}\eta^{2}\Delta_{2}\\
	\vdots & \vdots & \vdots & \ddots  & \vdots \\
	\eta^{-(m-2)}\eta^{m-2}\Delta_{m-2} & \eta^{-(m-3)}\eta^{m-3}\Delta_{m-3} & \dots &  \eta^{0}\eta^{0}\Delta_{0} & \eta^{m-1}\eta^{-(m-1)}\Delta_{m-1}\\
	\end{bmatrix}
	\]
	similarly we can factor out some powers of $\eta$ which is again $\pm 1$. So it is enough to show that the following matrix is non-degenerate:
	
	\[
	J'=
	\begin{bmatrix}
	\eta^{m-1}\Delta_{m-1} & \eta^{m-2}\Delta_{m-2} & \eta^{m-3}\Delta_{m-3} & \dots & \eta^{1}\Delta_{1} & \eta^{0}\Delta_{0}\\
	\eta^{0}\Delta_{0} & \eta^{m-1}\Delta_{m-1} & \eta^{m-2}\Delta_{m-2} & \dots & \eta^{2}\Delta_{2} & \eta^{1}\Delta_{1}\\
	\eta^{1}\Delta_{1} & \eta^{0}\Delta_{0} & \eta^{m-1}\Delta_{m-1} & \dots & \eta^{3}\Delta_{3} & \eta^{2}\Delta_{2}\\
	\vdots & \vdots & \vdots & \ddots  & \vdots & \vdots\\
	\eta^{m-2}\Delta_{m-2} & \eta^{m-3}\Delta_{m-3} & \eta^{m-4}\Delta_{m-4} & \dots &  \eta^{0}\Delta_{0} & \eta^{m-1}\Delta_{m-1}\\
	\end{bmatrix}
	\]
	
	We have the relation:
	
	$$m\eta^l\Delta_l=\sum_{j=0}^{m-1}\zeta^{lj}(e^{\eta\zeta^{j}\pi}-e^{-\eta\zeta^{j}\pi}),$$
	which actually defines a linear transformation from numbers $e^{\eta\zeta^{j}\pi}-e^{-\eta\zeta^{j}\pi}$ to numbers $m\eta^l\Delta_l$ , the matrix of this linear transformation is:
	$$\zeta^{lj}_{\{l,j=0,...m-1\}}.$$
	This is a Vandermonde matrix and the determinant is not zero because the $\zeta^l_{l=0,1,...,m-1}$ are all different.
	
	We see that:
	
	\[
	m^mJ'=VJ'',
	\]
	where $V$ is a Vandermonde matrix:
	\[
	V=
	\begin{bmatrix}
	1 & 1 & 1 & \dots & 1 & 1\\
	1 & \zeta^1 & \zeta^2 & \dots & \zeta^{m-2} & \zeta^{m-1}\\
	1 & \zeta^2 & \zeta^4 & \dots & \zeta^{2(m-2)} & \zeta^{2(m-1)}\\
	\vdots & \vdots & \vdots & \ddots  & \vdots & \vdots\\
	1 & \zeta^{m-1} & \zeta^{2(m-1)} & \dots &  \zeta^{(m-2)(m-1)} & \zeta^{(m-1)(m-1)}\\
	\end{bmatrix}
	\]
	Then after denoting $a_j=e^{\eta\zeta^{j}\pi}-e^{-\eta\zeta^{j}\pi}$ we see that:
	
	\begin{eqnarray*}
		J''&=&
		\begin{bmatrix}
			a_0 & a_0 & \dots & a_0 & a_0 & a_0\\
			\zeta^{m-1}a_1 & \zeta^{m-2}a_1 & \dots & \zeta^{2}a_1 & \zeta^{1}a_1 & a_1 \\
			\zeta^{2(m-1)}a_2 & \zeta^{2(m-2)}a_2 & \dots & \zeta^{4}a_2 & \zeta^{2}a_2 & a_2\\
			\vdots & \vdots & \vdots & \vdots  & \vdots & \vdots\\
			\zeta^{(m-1)(m-1)}a_{m-1} & \zeta^{(m-1)(m-2)}a_{m-1} & \dots &  \zeta^{2(m-1)}a_{m-1} & \zeta^{m-1}a_{m-1} & a_{m-1} \\
		\end{bmatrix}\\ \\
		&=&a_0a_1a_2...a_{m-1}
		\begin{bmatrix}
			1 & 1 & \dots & 1 & 1 & 1\\
			\zeta^{m-1} & \zeta^{m-2} & \dots & \zeta^{2} & \zeta^{1} & 1 \\
			\zeta^{2(m-1)} & \zeta^{2(m-2)} & \dots & \zeta^{4} & \zeta^{2} & 1\\
			\vdots & \vdots & \vdots & \vdots  & \vdots & \vdots\\
			\zeta^{(m-1)(m-1)} & \zeta^{(m-1)(m-2)} & \dots &  \zeta^{2(m-1)} & \zeta^{m-1} & 1 \\
		\end{bmatrix}
	\end{eqnarray*}
	
	Now it is clear that $a_0a_1...a_{m-1}\neq 0$, and all the Vandermonde matrices have non-zero determinant and the conclusion follows.
\end{proof}

  \bibliographystyle{amsalpha} 
  \bibliography{PolyReci}

\providecommand{\bysame}{\leavevmode\hbox to3em{\hrulefill}\thinspace}
\providecommand{\MR}{\relax\ifhmode\unskip\space\fi MR }
\providecommand{\MRhref}[2]{%
  \href{http://www.ams.org/mathscinet-getitem?mr=#1}{#2}
}
\providecommand{\href}[2]{#2}
\begin{thebibliography}{ASST01}

\bibitem[ASST01]{ADHIKARI20011}
S.D. Adhikari, N.~Saradha, T.N. Shorey, and R.~Tijdeman, \emph{Transcendental
  infinite sums}, Indagationes Mathematicae \textbf{12} (2001), no.~1, 1 -- 14.

\bibitem[Bun79]{Bundschuh1979}
Peter Bundschuh, \emph{Two remarks on transcendental numbers}, Monatshefte
  f{\"u}r Mathematik \textbf{88} (1979), no.~4, 293--304.

\bibitem[Cha03]{Chapman2003}
Robin Chapman, \emph{Evaluating $\zeta(2)$}.

\bibitem[D.H75]{LH}
D.H.Lehmer., \emph{Euler constants for arithmetical progressions}, Acta
  Arithmetica \textbf{27} (1975), 125--142.

\bibitem[MS10]{Murty2010}
M.~Ram Murty and N.~Saradha, \emph{Euler–lehmer constants and a conjecture of
  erd\text{\"o}s}, Journal of Number Theory \textbf{130} (2010), no.~12, 2671
  -- 2682.

\bibitem[PP07]{Pilerhood2007}
Kh.~Hessami Pilehrood and T.~Hessami Pilehrood, \emph{Infinite sums as linear
  combinations of polygamma functions}, Acta Arithmetica \textbf{130} (2007),
  231--254.

\end{thebibliography}





\end{document}